\documentclass[11pt,letterpaper]{amsart}
\usepackage{amsmath,amsthm,amssymb}
\usepackage[T1]{fontenc}
\usepackage[numbered]{bookmark}
\usepackage{changepage}
\usepackage[utf8]{inputenc}
\usepackage{graphics, graphicx, caption}
\usepackage{amsmath, amsfonts, amsthm, amssymb}
\usepackage{mathrsfs, mathtools, mathabx}
\usepackage{verbatim}
\usepackage{hyperref}
\usepackage{cleveref}
\usepackage{enumerate}
\usepackage{xcolor}
\usepackage{cite}

\newtheorem{theorem}{Theorem}[section]
\newtheorem{lemma}[theorem]{Lemma}
\newtheorem{proposition}[theorem]{Proposition}

\theoremstyle{remark}

\newtheorem{remark}[theorem]{Remark}

\numberwithin{equation}{section}

\newcommand{\Z}{\mathbb{Z}}
\newcommand{\N}{\mathbb{N}}

\newcommand{\R}{\mathbb{R}}

\newcommand{\A}{\mathcal{A}}

\title[Discrepancy in the regularity of conjugacies]{A surprising discrepancy in the regularity of conjugacies between generalized interval exchange transformations and their inverses at freezing}

\author[K.\ Fr\k{a}czek]{Krzysztof Fr\k{a}czek}
\address{Faculty of Mathematics and Computer Science, Nicolaus
	Copernicus University, ul. Chopina 12/18, 87-100 Toru\'n, Poland}
\email{fraczek@mat.umk.pl}

\author[\L.\ Kotlewski]{\L ukasz Kotlewski}
\address{Faculty of Mathematics and Computer Science, Nicolaus
	Copernicus University, ul. Chopina 12/18, 87-100 Toru\'n, Poland}
\email{kotlewskilukasz@mat.umk.pl}

\subjclass[2000]{37E05, 37C15, 37A05, 37A40, 37C45, 37C83, 37D35}

\begin{document}
	
	\maketitle
\begin{abstract}
Generalized interval exchange transformations (GIETs) are semi-conjugate to interval exchange transformations (IETs) when the Rauzy–Veech combinatorics is $\infty$-complete. When this semi-conjugacy is a homeomorphism, a fundamental problem is to understand the regularity of the conjugacy and its inverse. Contrary to the usual expectation that their H\"older regularities degenerate simultaneously, we exhibit a strongly asymmetric behavior.
For self-similar IETs of hyperbolic periodic type and a natural one-parameter central family of affine IET deformations obtained via a freezing (zero-temperature) limit, the conjugacy becomes arbitrarily irregular while its inverse remains uniformly Hölder. Using thermodynamic formalism for renormalization and zero-temperature limits, we obtain sharp asymptotics for Hausdorff dimensions of invariant and conformal measures and for the supremal H\"older exponents of the conjugacy and its inverse.

%Generalized interval exchange transformations (GIETs) arise as first return maps to Poincar\'e sections for flows on compact translation surfaces, and they are semi-conjugate to interval exchange transformations (IETs) whenever their (Rauzy--Veech) combinatorial rotation number is $\infty$-complete. A central problem, going back to rigidity questions in one-dimensional dynamics, is to quantify the regularity of the conjugacy and  its inverse when a semi-conjugacy is in fact a homeomorphism. In typical settings, one expects that the H\"older regularity of a conjugacy and its inverse deteriorate simultaneously; in this paper we exhibit a markedly different phenomenon.
%For  self-similar IETs of hyperbolic periodic type and a natural one-parameter \emph{central} family of AIET deformations (a freezing/zero-temperature limit), the conjugacy becomes arbitrarily irregular while its inverse stays uniformly regular. Using  thermodynamic formalism applied to renormalization and zero-temperature convergence for locally constant potentials, we obtain precise asymptotics for Hausdorff dimensions of the associated invariant and conformal measures, as well as for the supremal H\"older exponents of the conjugacy and its inverse.
\end{abstract}

%	\begin{abstract}
%		In this work we continue research on Hausdorff dimension of conformal measures for self-similar interval exchange transformations.  In the central case, we give precise formula of the limit of dimension as temperature goes to 0. Moreover we observe that under some nondegeneracy condition, H\"older exponent of the inverse conjugacy between AIET and IET stabilizes at positive value even though exponent of the conjugacy vanishes.
%	\end{abstract}

\section{Introduction}

Poincar\'e and Herman introduced two fundamental perspectives that supports one-dimensional dynamics: Poincar\'e’s use of transversal sections and return maps reduces smooth surface flows to piecewise interval dynamics, while Herman’s rigidity program stresses that dynamical classification requires quantitative control on the regularity of conjugacies.

Generalized interval exchange transformations (GIETs) are a modern realization of Poincaré’s approach: piecewise smooth, orientation-preserving interval bijections arising as return maps for flows on higher-genus surfaces. Their Lebesgue measure-preserving examples are interval exchange transformations (IETs), which are piecewise translations, while affine interval exchange transformations (AIETs) form an intermediate class that is piecewise linear with positive slopes.

Rauzy–Veech induction associates to a GIET  a renormalization scheme and a combinatorial rotation number. By a fundamental theorem of Yoccoz, a GIET $f$ with $\infty$-complete combinatorial rotation number is semi-conjugate to an IET $T$ with the same combinatorics, via a continuous map $h$. In the absence of wandering intervals, this semi-conjugacy is a conjugacy. This leads to the rigidity problem central to Herman’s program: determining the H\"older regularity of the conjugacy $h$ and its inverse $h^{-1}$.

While a $C^1$-diffeomorphism automatically has the inverse with the same H\"older regularity, no such symmetry holds for homeomorphisms. In low regularity regimes, the H\"older behavior of a conjugacy and its inverse may differ significantly.
The aim of this paper is to exhibit natural families of affine interval exchange transformations (AIETs) for which the gap between the supremal H\"older exponents $\mathfrak{H}$ of a conjugacy $h$ and its inverse $h^{-1}$ becomes arbitrarily large. Specifically, we construct one-parameter families in which $h^{-1}$ remains uniformly H\"older, while $\mathfrak{H}(h)$ tends to zero.

We work in the self-similar setting of IETs $T$ of hyperbolic periodic type and AIETs $f$ topologically conjugate to them, with central log-slope vectors $\omega$. For a fixed IET $T$ and a log-slope of  vector $\omega$, the corresponding class od AIETs is denoted $\operatorname{Aff}(T,\omega)$. In this setting, explicit formulas for $\mathfrak{H}(h)$, $\mathfrak{H}(h^{-1})$, and for the Hausdorff dimensions of the invariant and conformal measures were obtained recently in \cite{BFKT}, where $h$ and $h^{-1}$ arise as distribution functions of these measures.

We rewrite these formulas in a thermodynamic formalism framework for Rauzy–Veech renormalization, modeled by a shift of finite type $X$ and a locally constant potential $\Phi_\omega$  determined by the log-slope vector.
Let $m_{\Phi_\omega}$ denote the Gibbs measure on $X$ associated with the potential $\Phi_\omega$. Following the framework developed in \cite{Je} and\cite{ChGU}, we use the maximal ergodic average $\overline{\Phi}_\omega$ and the minimal ergodic average $\underline{\Phi}_\omega$, as well as the $\Phi_\omega$-maximizing $\overline{X}(\Phi_\omega)\subset X$ and $\Phi_\omega$-minimizing $\underline{X}(\Phi_\omega)\subset X$ subshifts.
These objects are designed to capture the structure governing the zero-temperature limit, that is, the asymptotic behavior of the family of Gibbs measures $m_{t\Phi_\omega}$ as the inverse temperature parameter $t\to+\infty$. Using results from \cite{ChGU} on zero-temperature limits (also referred to as the freezing regime) for locally constant potentials, we describe the asymptotic behavior of the supremal H\"older exponents $\mathfrak{H}(h_{t\omega})$ of the conjugacy  and its inverse $\mathfrak{H}(h^{-1}_{t\omega})$, as well as the Hausdorff dimensions of the invariant $\mu_{t\omega}$ and conformal measures $\nu_{t\omega}$, for families of AIETs $f_{t\omega}\in\operatorname{Aff}(T,t\omega)$ in the self-similar setting as $t\to+\infty$.
In this thermodynamic language, denoting by $h_{top}$  the topological entropy of the system, the main result of the paper can be stated as follows:

\begin{theorem}
Suppose that $T$ is an IET of hyperbolic periodic type and $\omega\in\R^{\mathcal A}$ is a log-slope vector of central type. Let us consider the family of AIETs $f_{t\omega}\in\operatorname{Aff}(T,t\omega)$, for $t\in\R$. Then,
\begin{gather*}
\lim_{t\to +\infty} \mathfrak{H}(h^{-1}_{t\omega}) = \frac{h_{top}(\overline{X}(\Phi_\omega))}{h_{top}(X)}=\lim_{t\to +\infty} \dim_H(\nu_{t\omega}),\quad\text{and}\\\
\lim_{t\to +\infty} t\,\mathfrak{H}(h_{t\omega}) =\frac{h_{top}(X)}{\overline{\Phi}_\omega-\underline{\Phi}_\omega} <\frac{h_{top}(X)}{\overline{\Phi}_\omega-\int_X\Phi_\omega\,d\,m_{\Phi_0}}=\lim_{t\to +\infty} t\dim_H(\mu_{t\omega}).
\end{gather*}
\end{theorem}
As a consequence, the supremal H\"older exponent of the conjugacy $h_{t\omega}$ between $f_{t\omega}$ and $T$ converges to zero as $t\to+\infty$, with a precisely determined decay rate. In contrast, the supremal H\"older exponent of the inverse conjugacy $h^{-1}_{t\omega}$ remains uniformly bounded from below by a constant which, as we show, is equal to $\frac{h_{top}(\overline{X}(\Phi_\omega))}{h_{top}(X)}$.
In the final section, we present an explicit example for which this constant is positive and compute its value. In this way, we construct a concrete one-parameter family of AIETs that exhibits the previously announced disparity between the regularity of the conjugacy and that of its inverse announced earlier.

\section{GIETs, AIETs and their conjugacy with IETs}\label{sec:conj}
One of the main objects investigated during the study of flows on compact surfaces of higher genus $g$ is \emph{generalized interval exchange transformations} (GIETs), that is, piecewise smooth (and increasing) bijections of the interval, with $d=2g+\kappa-1$ continuity intervals, where $\kappa$ is the number of singularities (fixed points) of the flow. They appear naturally as the first return maps to a Poincar\'e section of such a flow, i.e., as returns to an interval $I$ that is transversal to the  flow. With each GIET $f:I\to I$ we associate its combinatorial data, namely a finite alphabet $\mathcal{A}$ that labels the intervals before the rearrangement $(I_\alpha)_{\alpha\in\mathcal A}$ and the intervals after the rearrangement $(f(I_\alpha))_{\alpha\in\mathcal A}$, as well as a pair of bijections $\pi=(\pi_t,\pi_b)$, called a permutation, that describe the position of each interval before $\pi_t:\mathcal A\to\{1,\ldots,d\}$ and after the rearrangement $\pi_b:\mathcal A\to\{1,\ldots,d\}$. One of the key classes of GIETs are \emph{affine interval exchange transformations} (AIETs), which are piecewise linear GIETs with a positive constant derivative on each exchanged interval. If, in addition, the derivative is always equal to $1$, then we obtain \emph{interval exchange transformations} (IETs). The latter mappings arise when considering flows that preserve an area measure on surfaces.

Each interval exchange transformation $T:I\to I$ is uniquely determined by its combinatorial data $\pi$ and numerical data, that is, by the vector of lengths of the exchanged intervals $\lambda=(\lambda_\alpha)_{\alpha\in\mathcal A}\in\R^\mathcal{A}_{>0}$ given by $\lambda_\alpha=|I_\alpha|$ for $\alpha\in\mathcal A$. Therefore, the IET $T$ is identified with the pair $(\pi,\lambda)$.
In turn, each AIET $f:I\to I$ is uniquely determined by the pair $(\pi,\lambda)$ together with the vector of logarithms of the slopes $\omega=(\omega_\alpha)_{\alpha\in\mathcal A}$ given by $\omega_\alpha=\log Df|_{I_\alpha}$ for $\alpha \in \mathcal{A}$. Therefore, the AIET $f$ is identified with the triple $(\pi,\lambda,\omega)$.

One of the main tools to study GIETs is the \emph{Rauzy-Veech induction}. For every GIET $f:I\to I$ such that the last interval before the rearrangement $I_{\pi^{-1}_t(d)}$ is different from the last interval after the rearrangement $f(I_{\pi^{-1}_b(d)})$, we define $\mathcal{RV}(f):J\to J$
as the first return map to $J$ the longer of intervals $I\setminus I_{\pi^{-1}_t(d)}$ or $I\setminus f(I_{\pi^{-1}_b(d)})$. If $I\setminus I_{\pi^{-1}_t(d)}$ is the longer of the two, then we say that $f$ is of \emph{bottom type}, $\pi^{-1}_b(d)\in \mathcal{A}$ is a \emph{winner} of the induction step, and $\pi^{-1}_t(d)\in\mathcal{A}$ is a \emph{loser}. If, on the other hand, $I\setminus f(I_{\pi^{-1}_b(d)})$ is the longer of the two, then we say that $f$ is of \emph{top type}, $\pi^{-1}_t(d)$ is a \emph{winner} of the induction step, and $\pi^{-1}_b(d)$ is a \emph{loser}.  We say that $f$ is \emph{infinitely renormalizable} if this procedure can be continued indefinitely. Assuming that $f$ is infinitely renormalizable, we obtain a sequence of GIETs $\{\mathcal{RV}^k(f):I^{(k)}\to I^{(k)}\mid\,k\geq 0\}$ and a sequence of corresponding permutations $\{\pi^{(k)}\}_{k\geq 0}$.

A \emph{Rauzy graph} $\mathcal{RG(A)}$ is a directed labeled graph whose vertices are permutations of $\mathcal{A}$ and the edges lead from permutations to the possible images via $\mathcal{RV}$, which are labeled $t$ or $b$ depending on the type of renormalization.  For a given infinitely renormalizable GIET $f$, the infinite path $\gamma(f)=\gamma$ in $\mathcal{RG(A)}$ obtained by connecting the consecutive entries of $\{\pi^{(k)}\}_{k\in\N}$ is called the \emph{(combinatorial) rotation number} of $f$. Moreover, we say that $\gamma$ is \emph{$\infty$-complete} if every symbol $\alpha\in\A$ is a winner infinitely many times, see \cite{Yo05}.

As shown by Yoccoz \cite[Proposition 4]{Yo05}, the combinatorial rotation number of any infinitely renormalizable IET is $\infty$-complete. Moreover, by Corollary~5 in \cite{Yo05}, there exists at least one infinitely renormalizable IET associated with any $\infty$-complete path. Next, by Proposition~7 in\cite{Yo05}, any GIET $f$, whose combinatorial rotation number is $\infty$-complete, is semi-conjugate (via a continuous  map $h:I\to I$) to an IET $T$ with the same combinatorial rotation number; that is, $h\circ f=T\circ h$. The semi-conjugacy $h$ is a conjugacy (it is a homeomorphism) if the GIET has no wandering intervals, see Remark after the proof of Proposition~7 in \cite{Yo05}. One of the key questions in this theory is how regular (smooth) the conjugacies are, at the H\"older scale, and how regular their inverses are, provided they exist.
The main goal of the paper is to understand how large disparity between the regularity of a conjugacy and that of its inverse can be. In the case where the conjugacy $h$  is at least a $C^1$-diffeomorphism (which is the situation most commonly expected), a simple argument based on the formula for the derivative of the inverse function shows that the maximal regularities of $h$ and $h^{-1}$
are the same, so the aforementioned disparity does not occur. However, when $h$ is not differentiable, it is not immediately clear what the relationship is between the maximal regularities of $h$ and $h^{-1}$.
In this paper, we identify specific (one-parameter) families of GIETs (in fact, AIETs) for which the disparity between the regularity of $h$ and $h^{-1}$
can be made arbitrarily large as the parameter tends to infinity.

From now on, we will focus mainly on AIETs.
For every IET $T=(\pi,\lambda)$ and every $\omega\in \R^{\A}$, we denote by $\operatorname{Aff}(T,\omega)$ the set of AIETs semi-conjugated to $T$ with the log-slope $\omega$. As shown in \cite[Proposition 2.3]{MMY},
 $\operatorname{Aff}(T,\omega)$ is not empty if, and only if  $\langle \omega, \lambda \rangle = 0$. If the log-slope vector $\omega$ is additionally a central–stable vector (with respect to the Kontsevich–Zorich cocycle), then for almost every IET $T$, the semi-conjugacy between any $f\in\operatorname{Aff}(T,\omega)$ and $T$ is a homeomorphism, see \cite[Theorem 1]{TU} also for self-similar IETs. On the other hand, if the log-slope vector $\omega$ is an unstable vector, then $f\in\operatorname{Aff}(T,\omega)$  may have wandering intervals, which means that the semi-conjugacy $h$ is not invertible, see \cite{BHM} for self-similar IETs and \cite{MMY} for a.a.\ IETs.

The first significant step toward understanding the regularity of conjugacies is a result of Cobo \cite{Co}, which takes the form of the following dichotomy. For almost every IET $T$, if
$f\in\operatorname{Aff}(T,\omega)$ is conjugate to $T$ via $h:I\to I$ and the log-slope vector $\omega$ is a stable vector, then $h$ is a $C^1$-diffeomorphism; whereas if $\omega$ is not a stable vector, then
$h$ is not even an absolutely continuous function. The methods developed recently in \cite{BFKT} suggest that in the case where $\omega$ is not stable, neither the conjugacy nor its inverse (if it exists) is a H\"older function.
Thus, when studying AIETs with typical combinatorial rotation numbers, we are unlikely to observe the aforementioned disparity between the regularity of $h$ and $h^{-1}$, since the supremal H\"older exponents of both functions are simultaneously equal either to zero or to $1+\delta$.

This suggests that, in order to observe such a disparity in regularity, we should restrict ourselves to very specific IETs. From this point of view, self-similar IETs appear to be the most natural choice, and this is precisely what we do.

To conclude this section, it is worth mentioning recent breakthrough results of Ghazouani–Ulcigrai \cite{GU23,GU25} concerning the regularity of conjugacies between GIETs and IETs in the case of exchanges of a small number of intervals, namely $4$ or $5$. In this setting, for almost every IET, the conjugacy is a $C^{1+\delta}$-diffeomorphism, provided that the GIET satisfies the vanishing boundary operator condition. This latter condition arises naturally when starting from minimal flows on surfaces. A similar result, asserting $C^1$-conjugacy between arbitrary GIETs having the same and typical combinatorial rotation number (together with some additional natural assumptions), was recently obtained by Berk–Trujillo \cite{BT}. Both results also suggest the absence of a disparity in regularity in typical situations, even when we broaden our perspective to include GIETs and conjugacies between them.

\section{Self-similar IETs and corresponding AIETs}\label{sec:selfsim}
An IET $T=(\pi,\lambda)$ is \emph{self-similar} (or is of \emph{periodic type}), if there exists $n\ge 1$ such that $\mathcal{RV}^n(T):I^{(n)}\to I^{(n)}$ is conjugate to $T:I\to I$ via a rescaling $R:I^{(n)}\to I$ given by $R(x)=e^{\rho(T)}x$, where $\rho(T):=\log\frac{|I|}{|I^{(n)}|}$. Equivalently, this means that the combinatorial rotation number is a periodic path. Let us consider the related Rokhlin tower decomposition. Namely, for every $\alpha\in\A$ we denote by $q^{(n)}_{\alpha}$ the first return time of the interval $I^{(n)}_{\alpha}$ to $I^{(n)}$ via $T$, here $\{I^{(n)}_\alpha\}_{\alpha\in\mathcal A}$ are subintervals of $I^{(n)}$ exchanged by $\mathcal{RV}^n(T)$. Then, the whole interval $I$ is decomposed into the disjoint union of towers of intervals:
\begin{equation}\label{eq:towers}
I=\bigsqcup_{\alpha\in\A}\bigsqcup_{i=0}^{q^{(n)}_{\alpha}-1}T^i I^{(n)}_{\alpha}.
\end{equation}
Moreover, this partition is finer than the partition of $I$  into exchanged intervals $\{I_\alpha\}_{\alpha\in\mathcal A}$.
Let $M(T)=M=[M_{\alpha\beta}]_{\alpha,\beta\in\mathcal A}\in SL_{\A}(\Z)$ be the self-similarity matrix of $T$, that is,
\[M_{\alpha\beta}:=\# \{0\leq k < q^{(n)}_\alpha \mid T^k(I_{\alpha}^{(n)}) \subset I_{\beta} \}.\]
Then the vector $\lambda$ is a left Perron-Frobenius eigenvector and  $e^{\rho(T)}$ is the Perron-Frobenius eigenvalue of $M$.
If the matrix $M(T)$ is primitive, has exactly $g \geq 1$ real simple eigenvalues greater than (and smaller than) $1$ with different moduli and $\kappa-1$ unit ones, we say that $f$ is of \emph{hyperbolic periodic type}.
Recall that $g$ is the genus and $\kappa$ is the number of singularities of the translation surfaces associated to the permutation $\pi$. Moreover, every self-similar IET is uniquely ergodic.
If $T$ is of hyperbolic periodic type, then the right eigenvectors of $M$ form a basis of $\R^{\A}$; we have exactly $g$ expanding, $g$ contracting, and $\kappa-1$ invariant (fixed) eigenvectors.

In what follows we deal with AIETs $f$ belonging to $\operatorname{Aff}(T,\omega)$, where $\omega\in\R^{\A}$ is the log-slope vector of $f$ and $T$ is of hyperbolic periodic type.
Then $\langle \omega,\lambda\rangle=0$ and $\omega$ cannot have the right Perron-Frobenius eigenvector as a component in its basis decomposition. We say that $\omega$ is of
\begin{itemize}
 \item \emph{unstable type}, if in the basis decomposition it has at least one expanding (non-Perron-Frobenius) eigenvector,
 \item \emph{central-stable type}, if in the basis decomposition it has at least one invariant but no expanding eigenvectors.
 \item \emph{stable type}, if in the basis decomposition it has only contracting eigenvectors.
\end{itemize}
This classification is of fundamental importance for a complete understanding of the regularity of conjugacies, as was recently shown in \cite{BFKT}. For any function $h$ on an interval let $\mathfrak{H}(h)$ be the supremal  H\"older exponent, i.e.\
\[
\mathfrak{H}(h):=\sup\{\alpha\in\R_{\ge 0}\mid \text{$D^{\lfloor\alpha\rfloor}f$ is $\{\alpha\}$-H\"older} \}.
\]
\begin{theorem}[Theorems~1.3 and 1.6 in \cite{BFKT}] \label{thm: main2}
 Let $f\in \operatorname{Aff}(T,\omega)$ be an AIET semi-conjugated to a self-similar IET $T=(\pi,\lambda)$ of hyperbolic periodic type, with the log-slopes vector $\omega$. Let $h:I\to I$ be the semi-conjugacy between $f$ and $T$. Then
 \begin{enumerate}
 \item \label{numb:Thm 1.3 (1)}$\mathfrak{H}(h)=\mathfrak{H}(h^{-1})=1+\alpha$ for some $\alpha\in(0,1)\cup\{+\infty\}$, if $\omega$ is of stable type;
 \item \label{numb:Thm 1.3 (2)}$0< \mathfrak{H}(h),\mathfrak{H}(h^{-1})<1$, if $\omega$ is of central-stable type;
 \item \label{numb:Thm 1.3 (3)}$\mathfrak{H} (h)=0$, i.e. $h$ is not H\"older, if $\omega$ is of unstable type.
 \end{enumerate}
\end{theorem}
The precise formulas for the H\"older exponents in the central–stable case will be recalled later. In fact, as shown in \cite{BFKT}, the central–stable case reduces to the central case, that is, when the log-slope vector
$\omega$ is right invariant $M\,\omega=\omega$, and from now on we will focus only on this case.

For any vector $\omega\in\R^{\mathcal A}$ let $\phi_\omega:I\to\R$ be the associated piecewise constant map such that $\phi_\omega(x)=\omega_\alpha$ for $x\in I_\alpha$. For any map $\phi:I\to\R$ and $k\geq 0$, let $S_k\phi$ be its $k$-th Birkhoff sum, i.e.\ $S_k\phi(x)=\sum_{0\leq i<k}\phi(T^ix)$. Let
\begin{equation}\label{eq:sigma}
\Sigma:=\{(\alpha,i)\mid \alpha \in\mathcal A,\,0\leq i<q^{(n)}_{\alpha}\}
\end{equation}
be the alphabet parameterizing the partition of $I$ into intervals of the towers \eqref{eq:towers}. For any $(\alpha,i)\in\Sigma$ let $\beta=\beta(\alpha,i)\in\mathcal A$ be such that $T^iI^{(n)}_\alpha\subset I_\beta$, and let
\[\omega_{(\alpha,i)}:=S_{i}\phi_\omega|_{I^{(n)}_\alpha}=\sum_{0\leq j<i}\omega_{\beta(\alpha,j)}.\]
For any $M$-invariant vector $\omega$ let  $M(\omega)=[M_{\alpha\beta}(\omega)]_{\alpha,\beta\in\mathcal{A}}$ be  the matrix given by
\begin{equation}\label{eq:matxM}
M_{\alpha\beta}(\omega):=\sum_{\substack{0\leq i<q^{(n)}_\alpha\\ \beta(\alpha,i)=\beta}}e^{\omega_{(\alpha,i)}}.
\end{equation}
Then $M(0)=M$, and, since $M$ is primitive, the matrix $M(\omega)$ is primitive as well.
Denote by $\rho(\omega)$ the logarithm of Perron-Frobenius eigenvalue of $M(\omega)$, and let us consider the map $\rho:\R\to\R_{>0}$ given by $\rho(t):=\rho(t\omega)$.

For any $M$-invariant vector $\omega$ let $f\in\operatorname{Aff}(T,\omega)$. Let $\mu_\omega$ be a $f$-invariant probability Borel measure on $I$, and let $\nu_\omega$ be a $\phi_\omega$-conformal measure for $T$, that is, $\nu_\omega$ is a probability Borel measure on $I$ such that the measures $(T^{-1})_*\nu_\omega$ and $\nu_\omega$ are equivalent and the Radon-Nikodym derivative is $\frac{d(T^{-1})_*\nu_\omega}{d\nu_\omega}=e^{\phi_\omega}$.
As shown in \cite{BFKT}, both measures exist and are unique. Moreover, they play a key role in understanding the conjugacy $h_\omega:I\to I$ between $f$ and $T$. More precisely, $h_\omega$ is the distribution function of the measure $\mu_\omega$, while $h_\omega^{-1}$ is the distribution function of the measure $\nu_\omega$. Furthermore, in \cite{BFKT} explicit formulas for the Hausdorff dimensions of both measures were obtained.
\begin{theorem}[Theorem~5.2 in \cite{BFKT}]
Let $T$ be an IET of hyperbolic periodic type and let $\omega$ be a right $M$-invariant vector, where $M$ is the self-similarity matrix of $T$. Then, for every $t\in \R$, we have
\begin{equation}\label{eq:dims}
\dim_H(\mu_{t\omega})=\frac{\rho(0)}{\rho(t)-\rho'(0)t} \quad \text{and}\quad
	\dim_H(\nu_{t\omega})=\frac{\rho(t)-\rho'(t)t}{\rho(0)}.
\end{equation}
\end{theorem}
To determine the supremal H\"older exponents of the conjugacy and its inverse, we still need to describe two objects: a graph on $\Sigma$ and a certain edge potential on this graph.
Let $\mathcal{G}=(\Sigma,\mathcal{E})$ be a directed graph with the vertex set $\Sigma$ (defined by \eqref{eq:sigma}) and with edges given by
\[((\beta,j),(\alpha,i))\in\mathcal E \ \Longleftrightarrow\ \beta(\alpha,i)=\beta\ \Longleftrightarrow\ T^{i}I^{(n)}_\alpha\subset I_\beta.\]
For any right $M$-invariant vector $\omega$, let $\vartheta_\omega:\mathcal{E}\to \R_{\geq 0}$ be an edge potential given by
\[\vartheta_\omega((\beta,j),(\alpha,i)):=-\log\frac{\nu_\omega(T^{i}I^{(n)}_\alpha)}{\nu_\omega(I_\beta)}.\]
Next, consider non-negative numbers $\overline{\vartheta}_\omega$ and  $\underline{\vartheta}_\omega$, which are the largest and the smallest values of the averages of the function  $\vartheta_\omega$  when computed along elementary loops in the graph $\mathcal{G}$. A formal definition of both quantities will be given later, in a more general setting, in Section~\ref{section:therm}.
\begin{theorem}[Propositions~14.3 and~15.3 in \cite{BFKT}]
Let $T$ be an IET of hyperbolic periodic type and let $\omega$ be a right $M$-invariant vector, where $M$ is the self-similarity matrix of $T$. Then,
\begin{equation}\label{eq:Hexp}
\mathfrak{H}(h_{\omega})=\frac{\rho(0)}{\overline{\vartheta}_\omega}\in\big(0,\dim_H(\mu_{\omega})\big)  \text{ and }
	\mathfrak{H}(h^{-1}_{\omega})=\frac{\underline{\vartheta}_\omega}{\rho(0)}\in\big(0,\dim_H(\nu_{\omega})\big).
\end{equation}
\end{theorem}
The main goal of the paper is to precisely determine the asymptotic behavior of
\[\dim_H(\mu_{t\omega}),\quad \dim_H(\nu_{t\omega}),\quad \mathfrak{H}(h_{t\omega}),\quad \mathfrak{H}(h^{-1}_{t\omega})\]
as $t$ tends to infinity. To this end, we will use methods from the thermodynamic formalism for the shift of finite type determined by the graph $\mathcal G$, together with results from the classical paper \cite{ChGU} on freezing.

\section{Thermodynamical formalism and zero-temperature regime}\label{section:therm}
In this section, we present basic information on the thermodynamic formalism for shifts of finite type and on the behavior of Gibbs measures as the temperature tends to zero, relying mainly on \cite{Ru}, \cite{PU} and \cite{ChGU}.

Let $\Sigma$ be a finite alphabet  and $\mathcal{G}=(\Sigma,\mathcal{E})$ be a connected aperiodic directed graph with the vertex set $\Sigma$ and the edge set $\mathcal{E}$. Denote by $X=X(\mathcal G)\subset \Sigma^\Z$ the shift of finite type (SFT) determined by the graph $\mathcal G$. More precisely, $x=(x_n)_{n\in\Z}\in X$ if and only if $(x_n,x_{n+1})\in\mathcal{E}$ for any $n\in\Z$. Denote by $\sigma:X\to X$ the left shift map on $X$. This map is topologically mixing.

For any potential $\Phi:X\to\R$ and $k\geq 0$, let
\[S_k\Phi(x)=\sum_{0\leq i<k}\Phi(\sigma^ix).\]

We will deal with a specific class of potentials determined by edge potentials $\Phi:\mathcal{E}\to\R$ defined on the edge set.
 Then the associated potential $\Phi:X\to\R$ is given by $\Phi(x):=\Phi(x_0,x_1)$ for $x=(x_n)_{n\in\Z}\in X$. From now on, we will consider only potentials of this type.
For any $k\geq 1$, let $\operatorname{Per}_k(X)$ be the set of periodic sequences in $X$ of period $k$. This set is in one-to-one correspondence with the set of closed paths of length $k$ in the graph $\mathcal{G}$.
The \emph{topological pressure} of the potential $\Phi:X\to\R$ is defined as
\[P(\Phi,X)=\lim_{k\to\infty}\frac{1}{k}\log\sum_{x\in\operatorname{Per}_k(X)} e^{S_k\Phi(x)}.\]
Recall that the zero pressure $P(0,X)$ is equal to the \emph{topological entropy} $h_{top}(X)$ of the shift $\sigma:X\to X$.

Let $\mathcal{M}(\Phi)$ be a non-negative matrix $[\mathcal{M}_{ab}(\Phi)]_{a,b\in\Sigma}$ given by
\begin{equation}\label{eq:mathcalM}
\mathcal{M}_{ab}(\Phi)=\left\{\begin{array}{cl}
e^{\Phi(a,b)}&\text{if }(a,b)\in\mathcal{E}\\
0&\text{otherwise.}\end{array}\right.
\end{equation}
As the graph $\mathcal G$ is connected and aperiodic, the matrix $\mathcal{M}(\Phi)$ is primitive (its $m$th power is positive for some natural number $m$).
Moreover, $\mathcal{M}(0)$ is the adjacency matrix of the graph $\mathcal G$. Denote by $\varrho(\Phi)$ the logarithm of the Perron-Frobenius eigenvalue (spectral radius) of the matrix $\mathcal{M}(\Phi)$. Then
\begin{equation}\label{eq:PFpressure}
\varrho(\Phi)=\lim_{k\to\infty}\log\sqrt[k]{\operatorname{tr}\mathcal{M}^k(\Phi)}=\lim_{k\to\infty}\frac{1}{k}\log\sum_{x\in\operatorname{Per}_k(X)} e^{S_k\Phi(x)}=P(\Phi,X).
\end{equation}

For every $x\in X$ and integer numbers $m\leq n$, let $x[m,n]=x_m\ldots x_n\in\Sigma^{n-m+1}$. For any word $y=y_0\ldots y_{k-1}\in\Sigma^k$, let $[y]\subset X$ be the cylinder given by
\[[y]:=\{x\in X\mid x[0,k-1]=y\}.\]

A probability Borel $\sigma$-invariant measure $m_{\Phi}$ is a Gibbs measure for the potential $\Phi:X\to\R$ if there exists a constant $C>1$ such that
\[C^{-1}\leq \frac{m_{\Phi}\big(\big[x[0,k-1]\big]\big)}{\exp\big(S_k\Phi(x)-kP(\Phi,X)\big)}\leq C\quad\text{for all}\quad x\in X,\ k\in\N.\]
Since the potential is Hölder continuous, the Gibbs measure $m_\Phi$ is unique, and in this specific case, when the potential comes from an edge potential $\Phi:\mathcal{E}\to\R$, it is given as follows:
\begin{equation}\label{eq:Gibbs}
m_\Phi([y_0\ldots y_k])=e^{-k\varrho(\Phi)}\ell_{y_0}(\Phi)\Big(\prod_{0\leq i<k}\mathcal{M}_{y_iy_{i+1}}(\Phi)\Big)r_{y_k}(\Phi),
\end{equation}
for any cylinder $[y]=[y_0\ldots y_k]\subset X$, where $\ell(\Phi)$ is the left and $r(\Phi)$ is the right Perron-Frobenius eigenvector such that $|r(\Phi)|=1$ and $\ell(\Phi)^Tr(\Phi)=1$; both are positive. In the paper, we identify vectors with single-column matrices. The absolute value of a vector is the sum of the absolute values of its coordinates.

For any periodic orbit $x\in \operatorname{Per}_k(X)$, let
\[\Phi(x[0,k-1]):=S_k\Phi(x)=\sum_{0\leq i<k}\Phi(x_i,x_{i+1}),\]
and let
\[\overline{\Phi}:=\sup_{k\in\N}\max_{x\in\operatorname{Per}_k(X)}\frac{\Phi(x[0,k-1])}{k}=\max\Big\{\int_X\Phi\,d\,\lambda\mid \lambda\in\mathcal{P}(X,\sigma)\Big\},\]
where $\mathcal{P}(X,\sigma)$ is the simplex of Borel $\sigma$-invariant measures on $X$. Analogously, we can define
\[\underline{\Phi}:=\inf_{k\in\N}\min_{x\in\operatorname{Per}_k(X)}\frac{\Phi(x[0,k-1])}{k}=\min\Big\{\int_X\Phi\,d\,\lambda\mid \lambda\in\mathcal{P}(X,\sigma)\Big\}.\]
Then
\begin{equation}\label{eq:ouline}
\underline{\Phi}=-\overline{(-\Phi)}.
\end{equation}

A periodic orbit $x=(x_i)_{i\in\Z}\in X$ is called \emph{elementary}, if $x_0,x_1,\ldots,x_{k-1}$ are pairwise distinct, where $k=\operatorname{per}(x)$ is the least period of $x$. Let $\operatorname{Per}_{el}(X)$ be the (finite) set of elementary periodic orbits. Then
\[\overline{\Phi}=\max_{x\in \operatorname{Per}_{el}(X)}\frac{\Phi(x[0,\operatorname{per}(x)-1])}{\operatorname{per}(x)}.\]
An elementary periodic orbit $x$ is \emph{$\Phi$--maximizing}, if $\frac{\Phi(x[0,\operatorname{per}(x)-1])}{\operatorname{per}(x)}=\overline{\Phi}$.

Let us consider the $\Phi$--maximizing subshift $\overline{X}=\overline{X}(\Phi)\subset X$ consisting of $x\in X$ such that for every $n\in\N$ the word $x_n \, x_{n+1}$ occurs in an elementary $\Phi$--maximizing periodic orbit. Then $\overline{X}$ is a SFT determined by the graph $\overline{\mathcal G}=(\overline{\Sigma},\overline{\mathcal E})$, where $\overline{\Sigma}\subset\Sigma$ is the subset of the alphabet symbols which appear in elementary $\Phi$--maximizing periodic orbits, and $(a,b)\in\overline{\mathcal E}$ if the word $a \,b$ occurs in an elementary $\Phi$--maximizing periodic orbit.

In an analogous way, we can also define the $\Phi$-minimizing subshift $\underline{X}(\Phi)\subset X$ using $\underline{\Phi}$ instead of $\overline{\Phi}$.
The subshift  $\underline{X}(\Phi)$ is likewise a shift of finite type determined by a graph, which we will denote by $\underline{\mathcal G}=(\underline{\Sigma},\underline{\mathcal E})$
On the other hand,  the $\Phi$-minimizing subshift  can also be obtained as $\underline{X}(\Phi)=\overline{X}(-\Phi)$.

After \cite{ChGU}, let us consider two edge potentials $\Phi,\Psi :\mathcal{E}\to\R$ and the corresponding potentials $\Phi,\Psi :X\to\R$. For every $t\geq 0$, we deal with the potential $\Psi+t\Phi$ and the corresponding Gibbs measure $m_{\Psi+t\Phi}$. The main objective of the authors of article \cite{ChGU} was to investigate the asymptotic behavior of Gibbs measures as $t\to+\infty$, which from the thermodynamic perspective corresponds to the search for equilibrium states at zero-temperature. The following two results summarize the main results of \cite{ChGU} that we will use. In what follows, we will use the standard big $O$ notation.

\begin{theorem}[Theorem~1.1 in \cite{ChGU}]
Let $\mathcal{G}=(\Sigma,\mathcal{E})$ be a connected directed finite graph and let $X\subset \Sigma^\Z$ be the associated SFT. For any pair of edge potentials $\Phi,\Psi:\mathcal{E}\to\R$ the Gibbs measures $m_{\Psi+t\Phi}$, $t\geq 0$, converges weakly to a shift invariant measure $\nu$ on $X$ as $t\to+\infty$. Moreover, there exists $\eta>0$ such that for any cylinder $[y]\subset X$, we have
\begin{equation}\label{eq:Gibbsconv}
m_{\Psi+t\Phi}([y])=\nu([y])+O(e^{-\eta t}).
\end{equation}
\end{theorem}

\begin{remark}
If additionally $\Psi=0$, then, applying results of \cite{CG} (for more details on ergodic optimization we refer to the survey \cite{Je}), $\nu$ is a $\Phi$--maximizing measure, that is
\begin{equation}\label{eq:max}
\int_X\Phi\,d\,\nu=\max\Big\{\int_X\Phi\,d\,\lambda\mid \lambda\in\mathcal{P}(X,\sigma)\Big\}=\overline{\Phi}.
\end{equation}
\end{remark}

We would like to emphasize that article \cite{ChGU} also provides an effective algorithm for determining the limiting measure; however, for our purposes, the explicit form of the measure is not necessary. However, we will need more information about the asymptotic behavior of the Perron–Frobenius eigenvalues $\exp(\varrho(\Psi+t\Phi))$ of the matrices $\mathcal{M}(\Psi+t\Phi)$ as $t$  tends to infinity.

\begin{proposition}[Proposition~3 in \cite{ChGU}]
If additionally we have $\overline{\Phi}=0$ and $P(\Psi,\overline{X}(\Phi))=0$, then there exists $\eta>0$ such that
\[
\exp(\varrho(\Psi+t\Phi))=1+O(e^{-\eta t}).
\]
\end{proposition}

In fact, we will need the following version of this result.

\begin{proposition}\label{prop:rhozero}
For any pair of edge potentials $\Phi,\Psi:\mathcal{E}\to\R$ there exists $\eta>0$ such that
\begin{equation}
\varrho(\Psi+t\Phi)=P(\Psi,\overline{X}(\Phi))+\overline{\Phi}\,t+O(e^{-\eta t}).
\end{equation}
\end{proposition}

\begin{proof}
Let us consider the corrected edge potentials $\Phi_0:=\Phi-\overline{\Phi}$ and $\Psi_0:=\Psi-P(\Psi,\overline{X}(\Phi))$. By definition,
$\overline{\Phi}_0=0$ and $\overline{X}(\Phi_0)=\overline{X}(\Phi)$. Therefore,
\[P(\Psi_0,\overline{X}(\Phi_0))=P(\Psi,\overline{X}(\Phi_0))-P(\Psi,\overline{X}(\Phi))=0.\]
In view of Proposition~\ref{prop:rhozero}, we have $e^{\varrho(\Psi_0+t\Phi_0)}=1+O(e^{-\eta t})$. On the other hand,
$\mathcal{M}(\Psi+t\Phi)=e^{P(\Psi,\overline{X}(\Phi))+\overline{\Phi}\, t}\mathcal{M}(\Psi_0+t\Phi_0)$. It follows that
\[e^{\varrho(\Psi+t\Phi)}=e^{P(\Psi,\overline{X}(\Phi))+\overline{\Phi}\, t}e^{\varrho(\Psi_0+t\Phi_0)}=e^{P(\Psi,\overline{X}(\Phi))+\overline{\Phi}\, t}(1+O(e^{-\eta t})).\]
This gives,
\[\varrho(\Psi+t\Phi)=P(\Psi,\overline{X}(\Phi))+\overline{\Phi}\, t+\log(1+O(e^{-\eta t}))=P(\Psi,\overline{X}(\Phi))+\overline{\Phi}\, t+O(e^{-\eta t}).\]
\end{proof}

%To conclude this section, we focus on edge potentials $\Phi$ that depend only on the terminal vertex, that is, $\Phi(a,b)=\Phi(b)$ for a map $\Phi:\Sigma\to\R$.
The following result gives an explicit formula for the limit behavior of the derivative of $\varrho(t\Phi)$.

\begin{proposition}\label{prop:rhoder}
For any edge potential $\Phi:\mathcal{E}\to\R$ there exists $\eta>0$ such that
\[\frac{d}{dt}\varrho(t\Phi)=\int_X\Phi\,d\,\nu+O(e^{-\eta t})=\overline\Phi+O(e^{-\eta t}).\]
 \end{proposition}

 \begin{proof}
 By definition, $e^{\varrho(t\Phi)}=\ell^T(t\Phi)\mathcal{M}(t\Phi)r(t\Phi)$. It follows that
 \begin{align*}
 e^{\varrho(t\Phi)}\tfrac{d}{dt}\varrho(t\Phi)&=\frac{d}{dt}e^{\varrho(t\Phi)}=\ell^T(t\Phi)\frac{d}{dt}\mathcal{M}(t\Phi)r(t\Phi)\\
 &=
 \sum_{a,b\in\Sigma}\ell_a(t\Phi)\frac{d}{dt}\mathcal{M}_{ab}(t\Phi)r_b(t\Phi)\\
 &= \sum_{(a,b)\in\mathcal{E}}\ell_a(t\Phi)\frac{d}{dt}\mathcal{M}_{ab}(t\Phi)r_b(t\Phi).
 \end{align*}
 Note that the second equality follows easily by using the normalization $\ell(t\Phi)^Tr(t\Phi)=1$.
 By the definition of $\mathcal{M}$, for any pair $(a,b)\in \mathcal{E}$, we have
 \[\tfrac{d}{dt}\mathcal{M}_{ab}(t\Phi)=\mathcal{M}_{ab}(t\Phi)\Phi(a,b).\]
  In view of \eqref{eq:Gibbs}, we have $e^{-\varrho(t\Phi)}\ell_a(t\Phi)\mathcal{M}_{ab}(t\Phi)r_b(t\Phi)=m_{t\Phi}([a\,b])$ for any $[a\,b]\subset X$.
Therefore,
 \begin{align}\label{eq:dervar}
 \begin{aligned}
 \tfrac{d}{dt}\varrho(t\Phi)&=e^{-\varrho(t\Phi)}\sum_{(a,b)\in\mathcal{E}}\Phi(a,b)\ell_a(t\Phi)\mathcal{M}_{ab}(t\Phi)\,r_b(t\Phi)\\
& =\sum_{(a,b)\in\mathcal{E}}\Phi(a,b)m_{t\Phi}([a\,b])=\int_X\Phi\,d\,m_{t\Phi}.
\end{aligned}
\end{align}
 By \eqref{eq:Gibbsconv} and \eqref{eq:max}, this gives
 \[\tfrac{d}{dt}\varrho(t\Phi)=\sum_{(a,b)\in\mathcal{E}}\Phi(a,b)\nu([a\,b])+O(e^{-\eta t})=\int_X\Phi\,d\,\nu+O(e^{-\eta t})=\overline{\Phi}+O(e^{-\eta t}).\]
 \end{proof}

\section{Asymptotic behavior}
Let $T:I\to I$ be an IET of hyperbolic periodic type identified with a pair $(\pi,\lambda)$.
Let $\mathcal G=(\Sigma,\mathcal{E})$ be the corresponding directed graph defined in Section~\ref{sec:selfsim}, i.e.
\[\Sigma=\{(\alpha,i)\mid \alpha \in\mathcal A,\,0\leq i<q^{(n)}_{\alpha}\}\quad\text{and}\quad((\beta,j),(\alpha,i))\in\mathcal E  \Leftrightarrow \beta(\alpha,i)=\beta.\]
Denote by $X=X(\mathcal G)\subset \Sigma^\Z$ the shift of finite type (SFT) determined by the graph $\mathcal G$, and let $\sigma:X\to X$ be the left shift map on $X$.

For every right $M$-invariant vector $\omega\in\R^{\mathcal A}$ ($M$ is the self-similarity matrix of $T$), let $\Phi_\omega:\mathcal E\to\R$ be an edge potential given by
\[\Phi_\omega((\beta,j),(\alpha,i))=\log\frac{\nu_\omega(T^{i}I^{(n)}_\alpha)}{\nu_\omega(I^{(n)}_\alpha)}.\]
As $\nu_\omega$ is the unique $\phi_\omega$-conformal measure, we have
\[\Phi_\omega((\beta,j),(\alpha,i))=\log e^{S_i\phi_\omega|_{I^{(n)}_\alpha}}=\omega_{(\alpha,i)}.\]
It follows that
\begin{equation}\label{eq:phito}
\Phi_{t\omega}=t\Phi_\omega\quad\text{for any real}\ t.
\end{equation}

Recall that the vector $\omega$ also determines another edge potential $\vartheta_{\omega}:\mathcal{E}\to\R_{\geq 0}$ given by
\[\vartheta_\omega((\beta,j),(\alpha,i))=-\log\frac{\nu_\omega(T^{i}I^{(n)}_\alpha)}{\nu_\omega(I_\beta)}.\]
As in Section~\ref{sec:selfsim}, we can treat both potentials as functions defined on the shift space $X$.

\begin{lemma}
The map $\vartheta_\omega:X\to\R_{\geq 0}$ is cohomologous to the map $(\rho(\omega)-\Phi_\omega):X\to\R$, i.e.\ there exists a continuous map $g:X\to\R$ such that
$\vartheta_\omega(x)=\rho(\omega)-\Phi_\omega(x)+g(x)-g(\sigma x)$ for all $x\in X$.
\end{lemma}

\begin{proof}
Let $g:X\to\R$ be a map given by
\[g(x)=\log \nu_{\omega}(I_{\alpha_0})\quad\text{for any}\quad x=(x_k)_{k\in\Z}=(\alpha_k,i_k)_{k\in\Z}\in X.\]
As $\omega$ is an $M$-invariant vector, by Lemma~2.5 in \cite{BFKT}, we have $R_*(\nu_{\omega}|_{R^{-1}I})=e^{-\rho(\omega)}\nu_{\omega}$,
where $R:I^{(n)}\to I$ is the linear rescaling given by $R(x)=e^{\rho_T}x$ and establishes a conjugacy between IETs $\mathcal{RV}^n(T):I^{(n)}\to I^{(n)}$ and $T:I\to I$.
As $I^{(n)}_\alpha=R^{-1}I_\alpha$, it follows that $\nu_\omega(I^{(n)}_\alpha)=e^{-\rho(\omega)}\nu_{\omega}(I_\alpha)$. Hence, for any $x=(\alpha_k,i_k)_{k\in\Z}\in X$, we have
\begin{align*}
				(\vartheta_\omega-(\rho(\omega)-\Phi_{\omega}))(x)&= \log\frac{\nu_\omega(T^{i_1}I^{(n)}_{\alpha_1})}{\nu_\omega(I^{(n)}_{\alpha_1})}-\log \frac{\nu_\omega(T^{i_1}I^{(n)}_{\alpha_{i_1}})}{\nu_\omega(I_{\alpha_{i_0}})}-\rho(\omega)  \\
&= \log \nu_\omega(I_{\alpha_{i_0}})-\log\nu_\omega(I_{\alpha_1})=g(x)-g(\sigma x) .
			\end{align*}
\end{proof}

By the definition of the operators $\Phi\mapsto \overline{\Phi}$ and $\Phi\mapsto \underline{\Phi}$, in view of \eqref{eq:ouline} and \eqref{eq:phito},  the lemma above implies
\begin{gather}
\label{eq:vartheta}
\overline{\vartheta}_{t\omega}=\rho(t\omega)- t\underline{\Phi}_{\omega}\quad\text{and}\quad \underline{\vartheta}_{t\omega}=\rho(t\omega)- t\overline{\Phi}_{\omega}\quad\text{for any}\quad t\geq 0.
%\label{eq:vartheta-}
%\overline{\vartheta}_{t\omega}=\rho(t\omega)- t\overline{\Phi}_{\omega}\quad\text{and}\quad \underline{\vartheta}_{t\omega}=\rho(t\omega)- t\underline{\Phi}_{\omega}\quad\text{for any}\quad t\leq 0.
\end{gather}

\medskip

Let us consider a matrix $\mathcal{M}(\Phi_\omega)=[\mathcal{M}_{(\beta,j)(\alpha,i)}(\Phi_\omega)]_{(\beta,j),(\alpha,i)\in\Sigma}$ associated  (see \eqref{eq:mathcalM}) to the edge potential $\Phi_\omega$
\begin{equation}
\mathcal{M}_{(\beta,j)(\alpha,i)}(\Phi_{\omega})=\left\{\begin{array}{cl}
e^{\omega_{(\alpha,i)}}&\text{if }\beta(\alpha,i)=\beta\\
0&\text{otherwise.}\end{array}\right.
\end{equation}
According to the notation introduced in Section~\ref{section:therm}, $\varrho(\Phi_{\omega})$ is the logarithm of the Perron-Frobenius eigenvalue of the matrix $\mathcal{M}(\Phi_\omega)$.
In the next step, we will show that $\rho(t)=\rho(t\omega)=\varrho(\Phi_{t\omega})$, which makes it possible to apply results from the thermodynamic formalism, i.e.\ Proposition~\ref{prop:rhozero}~and~\ref{prop:rhoder}, to determine the asymptotics of $\rho(t)$ and its derivative. This, in turn, will allow us, via formulas \eqref{eq:dims} and \eqref{eq:Hexp}, to determine the asymptotic behavior of the Hausdorff dimensions of the invariant and conformal measures, as well as of the supremal H\"older exponents of the conjugacies and their inverses, as the parameter $t$ tends to infinity.

\begin{lemma}\label{lem:rhorho}
For any $M$-invariant vector $\omega\in\R^{\mathcal A}$, we have $\rho(\omega)=\varrho(\Phi_{\omega})$.
\end{lemma}

\begin{proof}
Recall that $\rho(\omega)$ is the logarithm of the Perron-Frobenius of the matrix $M(\omega)=[M_{\alpha\beta}(\omega)]_{\alpha,\beta\in\mathcal A}$ defined by \eqref{eq:matxM}.
Therefore, it suffices to show that the Perron–Frobenius eigenvalues of both matrices $M(\omega)$ and $\mathcal{M}(\Phi_\omega)$ are equal. To begin with, we will exhibit two nonnegative matrices $E$ and $D(\omega)$ that reveal a direct connection between $M(\omega)$ and $\mathcal{M}(\Phi_\omega)$, that is, $M^T(\omega)=D(\omega)\, E$ and $\mathcal{M}(\Phi_\omega)=E\, D(\omega)$. Indeed, let \[E=[E_{(\beta,j),\alpha}]_{(\beta,j)\in\Sigma,\alpha\in\mathcal A}\quad\text{and}\quad D(\omega)=[D_{\beta,(\alpha,i)}(\omega)]_{\beta\in\mathcal{A},(\alpha,i)\in\Sigma}\]
given by
\[D_{\beta,(\alpha,i)}(\omega)=
		\left\{
		\begin{array}{cc}
			e^{\omega_{(\alpha,i)}}&\text{if }\beta(\alpha,i)=\beta,\\
			0 &\text{if }\beta(\alpha,i)\neq\beta;
		\end{array}
		\right.\quad\text{and}\quad
		E_{(\beta,j),\alpha}=
		\left\{
		\begin{array}{cc}
			1&\text{if }\beta=\alpha,\\
			0 &\text{if }\beta \neq\alpha.
		\end{array}
		\right.\]
Then
\begin{gather*}
				(E\,D(\omega))_{(\beta,j),(\alpha,i)}=\sum_{\gamma\in\mathcal{A}} E_{(\beta,j),\gamma}\,D_{\gamma,(\alpha,i)}(\omega)=D_{\beta,(\alpha,i)}(\omega)= \mathcal M(\Phi_{\omega})_{(\beta,j),(\alpha,i)},\\
				(D(\omega)\,E)_{\beta\alpha}=\sum_{(\gamma,k)\in\Sigma} D_{\beta,(\gamma,k)}(\omega)E_{(\gamma,k),\alpha}=\sum_{\substack{0\leq k<q^{(n)}_\alpha\\
\beta(\alpha,k)=\beta}}e^{\omega_{(\alpha,k)}} = M_{\alpha\beta}(\omega),
			\end{gather*}
and hence $\mathcal{M}(\Phi_\omega)=E\, D(\omega)$ and $M^T(\omega)=D(\omega)\, E$.

Let $\nu=(\nu_{\alpha})_{\alpha\in\mathcal A}, \theta=(\theta_\alpha)_{\alpha\in\mathcal A} \in \mathbb{R}_{>0}^{\mathcal{A}}$ be respectively the left and the right Perron-Frobenius eigenvectors of the matrix $M(\omega)$, with the Perron-Frobenius eigenvalue $e^{\rho(\omega)}$. Then the vectors $\ell:= D^T(\omega)\theta\in\R_{\geq 0}^{\Sigma} $ and $r:= E\nu\in\R_{\geq 0}^{\Sigma}$ are the left and the right Perron-Frobenius eigenvectors of the matrix $\mathcal{M}(\Phi_\omega)$. Indeed,
\begin{align*}
\mathcal{M}(\Phi_\omega) r&=\mathcal{M}(\Phi_\omega) E\nu= ED(\omega)E\nu = E M^T \nu = e^{\rho(\omega)} E\nu= e^{\rho(\omega)}r,\\
\ell^T\mathcal{M}(\Phi_\omega)&=\theta^T D(\omega) \mathcal{M}(\Phi_\omega) =  \theta^T D(\omega)ED(\omega)\\
& = \theta^T M^T(\omega)D(\omega)=e^{\rho(\omega)}\theta^TD(\omega)=e^{\rho(\omega)}\ell^T.
\end{align*}
It follows that the Perron–Frobenius eigenvalues of both matrices $M(\omega)$ and $\mathcal{M}(\Phi_\omega)$ are equal, which completes the proof.
\end{proof}

We are now ready to prove the main results of the article. Suppose that $T=(\pi,\lambda)$ is an IET of hyperbolic periodic type, with its self-similarity matrix $M$. Let $\mathcal{G}=(\Sigma,\mathcal{E})$ be the associated graph and $X\subset \Sigma^\Z$ the corresponding shift of finite type. Let $\omega\in\R^{\mathcal A}$ be an $M$-invariant vector, with the corresponding edge potential $\Phi_\omega:\mathcal{E}\to\R$.
Denote by $m_{\Phi_\omega}$ the unique Gibbs measure on $X$ related to the potential $\Phi_\omega$.
Let $\overline{X}(\Phi_\omega)\subset X$ be the $\Phi_\omega$-maximizing shift of finite type.

Let us consider an AIET $f_\omega\in\operatorname{Aff}(T,\omega)$ conjugated to the IET $T$ via $h_{\omega}:I\to I$. Denote by $\mu_\omega$ the unique $f_\omega$-invariant measure, and by $\nu_{\omega}$ the unique $\phi_\omega$-conformal measure.
	
\begin{theorem}\label{dimension_conformal_limit}
Suppose that $T=(\pi,\lambda)$ is an IET of hyperbolic periodic type, with its self-similarity matrix $M$. For any $M$-invariant vector $\omega\in\R^{\mathcal A}$, there exists $\eta>0$ such that for $t\geq 0$, we have
\begin{gather*}
\dim_H(\nu_{t\omega})=\frac{h_{top}(\overline{X}(\Phi_\omega))+O(e^{-\eta t})}{h_{top}(X)},\\
\dim_H(\mu_{t\omega})=\frac{h_{top}(X)}{(\overline{\Phi}_\omega-\int_X\Phi_\omega\,d\,m_{\Phi_0})t+h_{top}(\overline{X}(\Phi_\omega))+O(e^{-\eta t})}.
\end{gather*}
In particular,
\[\lim_{t\to +\infty} \dim_H(\nu_{t\omega}) = \frac{h_{top}(\overline{X}(\Phi_\omega))}{h_{top}(X)},\quad \lim_{t\to +\infty} t\dim_H(\mu_{t\omega}) =\frac{h_{top}(X)}{\overline{\Phi}_\omega-\int_X\Phi_\omega\,d\,m_{\Phi_0}}.\]
\end{theorem}		
	
\begin{proof}
By \eqref{eq:dims}, we have
\begin{equation*}
\dim_H(\mu_{t\omega})=\frac{\rho(0)}{\rho(t)-\rho'(0)t} \quad \text{and}\quad
	\dim_H(\nu_{t\omega})=\frac{\rho(t)-\rho'(t)t}{\rho(0)},
\end{equation*}
where $\rho(t)=\rho(t\omega)$ is the Perron-Frobenius eigenvalue of the matrix $M(t\omega)$. In view of Lemma~\ref{lem:rhorho}, $\rho(t)=\varrho(\Phi_{t\omega})$, where $\varrho(\Phi_{t\omega})$ is the Perron-Frobenius eigenvalue of the matrix $\mathcal M(\Phi_{t\omega})$. As $\Phi_{t\omega}=t\Phi_{\omega}$, by Proposition~\ref{prop:rhozero} and \ref{prop:rhoder}, there exists $\eta>0$ such that
\begin{align}\label{eq:rhot}
\begin{aligned}
\rho(t)&=\varrho(t\Phi_\omega)=P(0,\overline{X}(\Phi_\omega))+\overline{\Phi}_\omega\,t+O(e^{-\eta t})\\
&=h_{top}(\overline{X}(\Phi_\omega))+\overline{\Phi}_\omega\,t+O(e^{-\eta t}),
\end{aligned}
\end{align}
\[\rho'(t)=\frac{d}{dt}\varrho(t\Phi_\omega)=\overline\Phi_\omega+O(e^{-\eta t}).\]
Moreover, by \eqref{eq:PFpressure} and \eqref{eq:dervar},
\begin{equation}\label{eq:rho0}
\rho(0)=P(0,X)=h_{top}(X)\quad\text{and}\quad\rho'(0)=\int_X\Phi_\omega\,d\,m_{\Phi_0}.
\end{equation}
It follows that
\begin{align*}
\dim_H(\mu_{t\omega})&=\frac{h_{top}(X)}{h_{top}(\overline{X}(\Phi_\omega))+\overline{\Phi}_\omega\,t+O(e^{-\eta t})-t\int_X\Phi_\omega\,d\,m_{\Phi_0}},
\\
	\dim_H(\nu_{t\omega})&=\frac{h_{top}(\overline{X}(\Phi_\omega))+\overline{\Phi}_\omega\,t+O(e^{-\eta t})-t\big(\overline\Phi_\omega+O(e^{-\eta t})\big)}{h_{top}(X)}.
\end{align*}
In this way, by slightly decreasing the constant $\eta>0$, we obtain the main claim of the theorem.
				\end{proof}
			
\begin{theorem}\label{Holder_conformal_limit}
Suppose that $T=(\pi,\lambda)$ is an IET of hyperbolic periodic type, with its self-similarity matrix $M$. For any $M$-invariant vector $\omega\in\R^{\mathcal A}$, there exists $\eta>0$ such that for $t\geq 0$, we have
\begin{gather}
\label{eq:Hol1}\mathfrak{H}(h^{-1}_{t\omega})=\frac{h_{top}(\overline{X}(\Phi_\omega))+O(e^{-\eta t})}{h_{top}(X)}.\\
\label{eq:Hol2}\mathfrak{H}(h_{t\omega})=\frac{h_{top}(X)}{(\overline{\Phi}_\omega-\underline{\Phi}_\omega)t+h_{top}(\overline{X}(\Phi_\omega))+O(e^{-\eta t})}.
\end{gather}
In particular,
\[\lim_{t\to +\infty} \mathfrak{H}(h^{-1}_{t\omega}) = \frac{h_{top}(\overline{X}(\Phi_\omega))}{h_{top}(X)}\quad\text{and}\quad \lim_{t\to +\infty} t\,\mathfrak{H}(h_{t\omega}) =\frac{h_{top}(X)}{\overline{\Phi}_\omega-\underline{\Phi}_\omega}.\]
\end{theorem}

\begin{proof}
By \eqref{eq:Hexp} and \eqref{eq:vartheta}, we have
\begin{equation}\label{eq:formHex}
\mathfrak{H}(h_{t\omega})=\frac{\rho(0)}{\overline{\vartheta}_{t\omega}}=\frac{\rho(0)}{\rho(t)- t\underline{\Phi}_{\omega}} \quad \text{and}\quad
	\mathfrak{H}(h^{-1}_{t\omega})=\frac{\underline{\vartheta}_{t\omega}}{\rho(0)}=\frac{\rho(t)- t\overline{\Phi}_{\omega}}{\rho(0)}.
\end{equation}
In view of \eqref{eq:rhot} and \eqref{eq:rho0}, this implies \eqref{eq:Hol1} and \eqref{eq:Hol2}.
\end{proof}

\begin{remark}
To understand the asymptotics of the quantities under consideration in the case when $t\to-\infty$, it suffices to apply the last two theorems with $\omega$ replaced by $-\omega$.
Since $\Phi_{-\omega}=-\Phi_\omega$, the minimizing and maximizing subshifts exchange their roles. Summarizing, we then obtain:
\begin{gather*}
\lim_{t\to -\infty} \mathfrak{H}(h^{-1}_{t\omega}) = \frac{h_{top}(\underline{X}(\Phi_\omega))}{h_{top}(X)}=\lim_{t\to -\infty} \dim_H(\nu_{t\omega}),\quad\text{and}\\\
\lim_{t\to -\infty} |t|\,\mathfrak{H}(h_{t\omega}) =\frac{h_{top}(X)}{\overline{\Phi}_\omega-\underline{\Phi}_\omega} <\frac{h_{top}(X)}{\int_X\Phi_\omega\,d\,m_{\Phi_0}-\underline{\Phi}_\omega}=\lim_{t\to -\infty} |t|\dim_H(\mu_{t\omega}).
\end{gather*}
\end{remark}

\begin{remark}\label{rem:mono}
Theorems~\ref{dimension_conformal_limit} and \ref{Holder_conformal_limit} provide us with information about the limiting behavior of the Hausdorff dimension of the measures and of the supremal H\"older exponent of the conjugacies and their inverses as $t$ tends to infinity. In addition, we know that this behavior is monotonic (with respect to the parameter $t$) both on the positive and on the negative half-line. Indeed, the monotonicity of the Hausdorff dimensions was established in Theorem~5.3 of \cite{BFKT}. The monotonicity of the supremal H\"older exponents follows easily from \eqref{eq:formHex} and \eqref{eq:dervar}, since
\begin{align*}
\tfrac{d}{dt}(\rho(t)-t\underline{\Phi}_\omega)=\rho'(t)-\underline{\Phi}_\omega=\int_X \Phi_\omega\,d\,m_{t\Phi_\omega}-\underline{\Phi}_\omega\geq 0,\\
\tfrac{d}{dt}(\rho(t)-t\overline{\Phi}_\omega)=\rho'(t)-\overline{\Phi}_\omega=\int_X \Phi_\omega\,d\,m_{t\Phi_\omega}-\overline{\Phi}_\omega\leq 0.
\end{align*}
\end{remark}

			\section{Example}
Based on Theorem~\ref{Holder_conformal_limit} and Remark~\ref{rem:mono}, we know that for any IET $T$ of hyperbolic periodic type, if $\omega$ is an invariant vector of its self-similarity matrix $M$ and $f_{t\omega}\in\operatorname{Aff}(T,t\omega)$, then the supremal H\"older exponent of the conjugacy $h_{t\omega}$ between $f_{t\omega}$ and $T$ converges monotonically to zero as $t\to\infty$, and we know the precise rate of its decay. On the other hand, the supremal H\"older exponent of the inverse conjugacy $h^{-1}_{t\omega}$ is bounded from below by the constant
$\min\big\{\frac{h_{top}(\overline{X}(\Phi_\omega))}{h_{top}(X)}, \frac{h_{top}(\underline{X}(\Phi_\omega))}{h_{top}(X)}\big\}$.
In this section, we give a concrete example for which this constant is positive and compute its value. In this way, we construct an explicit one-parameter family of AIETs for which we observe the disparity between the regularity of the conjugacy and that of its inverse announced in the introduction.

We begin with an example of a $5$-IET of hyperbolic periodic type, which was already considered in \cite{BF}. We used Python for all computations.
Consider $\mathcal A:=\{A,B,C,D,E\}$ and the symmetric permutation \begin{equation*}
				\pi = \Bigl(\begin{matrix}
					A& B & C & D & E \\
					E& D & C& B & A
				\end{matrix}\Bigr),
			\end{equation*}
and a closed path in the Rauzy graph starting at $\pi$ and following the edges labeled by: ttbbtbtbbbtb. The corresponding self-similarity matrix is of the form
			{\tiny\begin{equation*}
				M=
				\left(\begin{matrix}
					1& 1 & 0 & 0 & 2 \\
					1& 2 & 0& 0 & 3\\
					1& 0 & 2 & 0 & 2 \\
					1 &0 &3 & 2 &2 \\
					1 &0 &2 &1 & 2
				\end{matrix}\right)
			\end{equation*}}
and is primitive ($M^2$ is positive) and of hyperbolic type; here $g=2$ and $\kappa=2$. The Perron-Frobenius eigenvalue of $M$ is equal to
\[\theta_0:=2+\frac{1}{2}\sqrt{3}+\frac{1}{2}\sqrt{15+8\sqrt{3}}\approx 5.55,\]
and
\begin{align*}
\lambda&=
\Big(
\sqrt{3},
\frac{3}{2}-\sqrt{3}+\sqrt{15+8\sqrt{3}}-\frac{1}{2}\sqrt{3}\,\sqrt{15+8\sqrt{3}},\\
&\quad
-1+\frac{3}{2}\sqrt{3}-\frac{3}{2}\sqrt{15+8\sqrt{3}}+\sqrt{3}\,\sqrt{15+8\sqrt{3}},
1,
\frac{1}{2}\sqrt{3}+\frac{1}{2}\sqrt{15+8\sqrt{3}}
\Big)
\end{align*}
is its left Perron-Frobenius eigenvector. Then $T=(\pi,\lambda)$ is an IET of hyperbolic periodic type (with period $n=12$) with the self-similarity matrix $M$. As $\kappa=2$, the right $M$-invariant space is one-dimensional and generated by $\omega:=(-1,-2,-1,2,1)$. Then, the associated  matrix $M(t\omega)$ is given by
			{\tiny\begin{equation*}
				\left(\begin{matrix}
					1& 1 & 0 & 0 & e^{-t}+e^{-2t} \\
					1& 1+e^{-t} & 0& 0 & e^{-t}+e^{-2t}+e^{-3t}\\
					1& 0 & 1+e^{-t} & 0 & e^{-t}+e^{-2t} \\
					1 &0 &e^{2t}+e^{t}+1 & 1+e^{-t} &e^{t}+e^{-t} \\
					1 &0 &e^{t}+1 &e^{-t} & 1+e^{-t}
				\end{matrix}\right).
			\end{equation*}}
For any real $t$, let us consider an AIET $f_{t\omega}\in\operatorname{Aff}(T,t\omega)$. Then the length of intervals exchanged by $f_{t\omega}$ are given by coordinates of the left Perron-Frobenius eigenvector  of the matrix $M(t\omega)$.
			For $t=1$, its left normalized Perron-Frobenius eigenvector is approximately equal to
			\[(0.258620,
			0.103498,
			0.415430,
			0.028547,
			0.193905).\]
			
As $q^{(n)}_A=4$, $q^{(n)}_B=6$, $q^{(n)}_C=5$, $q^{(n)}_D=8$, and $q^{(n)}_E=6$, the extended alphabet is
\begin{gather*}
\Sigma=\big\{A0 ,\, A1 ,\, A2 ,\, A3 ,\, B0 ,\, B1 ,\, B2 ,\, B3 ,\, B4 ,\, B5 ,\, C0 ,\, C1 ,\, C2 ,\, C3 ,\, C4 ,\\
D0 ,\, D1 ,\, D2 ,\, D3 ,\, D4 ,\, D5 ,\, D6 ,\, D7 ,\, E0 ,\, E1 ,\, E2 ,\, E3 ,\, E4 ,\, E5\}.
\end{gather*}
			\begin{table}[h!]
				\centering
				\resizebox{\textwidth}{!}{
					$\begin{array}{c|cccc|cccccc|ccccc|cccccccc|cccccc}
						& A0 & A1 & A2 & A3 & B0 & B1 & B2 & B3 & B4 & B5 & C0 & C1 & C2 & C3 & C4 & D0 & D1 & D2 & D3 & D4 & D5 & D6 & D7 & E0 & E1 & E2 & E3 & E4 & E5 \\ \hline
						\mathbf{A} & 1 & 0 & 0 & 0 & 1 & 0 & 0 & 0 & 0 & 0 & 1 & 0 & 0 & 0 & 0 & 1 & 0 & 0 & 0 & 0 & 0 & 0 & 0 & 1 & 0 & 0 & 0 & 0 & 0 \\
						\mathbf{B} & 0 & 0 & 1 & 0 & 0 & 0 & 1 & 0 & e^{-t} & 0 & 0 & 0 & 0 & 0 & 0 & 0 & 0 & 0 & 0 & 0 & 0 & 0 & 0 & 0 & 0 & 0 & 0 & 0 & 0 \\
						\mathbf{C} & 0 & 0 & 0 & 0 & 0 & 0 & 0 & 0 & 0 & 0 & 0 & 0 & 1 & e^{-t} & 0 & 0 & 0 & 1 & 0 & e^{t} & 0 & e^{2t} & 0 & 0 & 0 & 1 & 0 & e^{t} & 0 \\
						\mathbf{D} & 0 & 0 & 0 & 0 & 0 & 0 & 0 & 0 & 0 & 0 & 0 & 0 & 0 & 0 & 0 & 0 & 0 & 0 & e^{-t} & 0 & 1 & 0 & 0 & 0 & 0 & 0 & e^{-t} & 0 & 0 \\
						\mathbf{E} & 0 & e^{-t} & 0 & e^{-2t} & 0 & e^{-t} & 0 & e^{-2t} & 0 & e^{-3t} & 0 & e^{-t} & 0 & 0 & e^{-2t} & 0 & e^{-t} & 0 & 0 & 0 & 0 & 0 & e^{t} & 0 & e^{-t} & 0 & 0 & 0 & 1 \\ \hline
					\end{array}$
				}
				\caption{Simplified version of the matrix $\mathcal{M}(t\Phi_\omega)$}\label{tab1}
			\end{table}
Since the values of the edge potential  $\Phi_\omega((\beta,j),(\alpha,i))$ do not depend on the coordinate $j$, all rows of the matrix $\mathcal{M}(t\Phi_\omega)$ indexed by $(\beta,j)$, for $0\leq j< q^{(n)}_\beta$,
are identical. Therefore, in order to simplify the presentation of the matrix $\mathcal{M}(t\Phi_\omega)$ in Table~\ref{tab1}, we do not repeat the identical rows. In this way, we obtain a simplified matrix whose rows are indexed by the elements of the alphabet $\mathcal{A}$.

The next step (this time computationally demanding) is to determine $\overline{\Phi}_\omega=0$, $\underline{\Phi}_\omega=-1$, the $\Phi_\omega$-maximizing and $\Phi_\omega$-minimizing elementary loops in the graph $\mathcal{G}=(\Sigma,\mathcal{E})$, as well as to identify the corresponding $\Phi_\omega$-maximizing and $\Phi_\omega$-minimizing subshifts $\overline{X}(\Phi_\omega)$ and $\underline{X}(\Phi_\omega)$ together with their graphs $\overline{\mathcal{G}}=(\overline{\Sigma},\overline{\mathcal{E}})$ and $\underline{\mathcal{G}}=(\underline{\Sigma},\underline{\mathcal{E}})$.
The outcome of these computations consists of the reduced sets of vertices
\begin{gather*}
\overline{\Sigma}=\big\{A0 ,\, A1 ,\, A2 ,\, B0 ,\, B1 ,\, B2 ,\, C0 ,\, C1 ,\, C2 ,\, D5 ,\, D6 ,\, D7 ,\, E3 ,\, E4 ,\, E5\big\},\\
\underline{\Sigma}=\big\{A2 ,\, A3 ,\, B4 ,\, B5 ,\, C3 ,\, C4 ,\, D0 ,\, D1 ,\, D2 ,\, D3 ,\, E0 ,\, E1 ,\, E2 ,\, E3\big\},
\end{gather*}
and the adjacency matrices of the graphs $\overline{\mathcal{G}}$ and $\underline{\mathcal{G}}$ presented in Tables~\ref{tab2} and ~\ref{tab3} respectively.
			\begin{table}[ht]
				\centering
				{\tiny
				\setlength{\tabcolsep}{2.5pt}
				\begin{tabular}{r|*{15}{c}}
					& A0 & A1 & A2 & B0 & B1 & B2 & C0 & C1 & C2 & D5 & D6 & D7 & E3 & E4 & E5 \\ \hline
					A0 & 1 & 0 & 0 & 1 & 0 & 0 & 1 & 0 & 0 & 0 & 0 & 0 & 0 & 0 & 0 \\
					A1 & 1 & 0 & 0 & 1 & 0 & 0 & 1 & 0 & 0 & 0 & 0 & 0 & 0 & 0 & 0 \\
					A2 & 1 & 0 & 0 & 1 & 0 & 0 & 1 & 0 & 0 & 0 & 0 & 0 & 0 & 0 & 0 \\
					B0 & 0 & 0 & 1 & 0 & 0 & 1 & 0 & 0 & 0 & 0 & 0 & 0 & 0 & 0 & 0 \\
					B1 & 0 & 0 & 1 & 0 & 0 & 1 & 0 & 0 & 0 & 0 & 0 & 0 & 0 & 0 & 0 \\
					B2 & 0 & 0 & 1 & 0 & 0 & 1 & 0 & 0 & 0 & 0 & 0 & 0 & 0 & 0 & 0 \\
					C0 & 0 & 0 & 0 & 0 & 0 & 0 & 0 & 0 & 1 & 0 & 1 & 0 & 0 & 1 & 0 \\
					C1 & 0 & 0 & 0 & 0 & 0 & 0 & 0 & 0 & 1 & 0 & 1 & 0 & 0 & 1 & 0 \\
					C2 & 0 & 0 & 0 & 0 & 0 & 0 & 0 & 0 & 1 & 0 & 1 & 0 & 0 & 1 & 0 \\
					D5 & 0 & 0 & 0 & 0 & 0 & 0 & 0 & 0 & 0 & 1 & 0 & 0 & 1 & 0 & 0 \\
					D6 & 0 & 0 & 0 & 0 & 0 & 0 & 0 & 0 & 0 & 1 & 0 & 0 & 1 & 0 & 0 \\
					D7 & 0 & 0 & 0 & 0 & 0 & 0 & 0 & 0 & 0 & 1 & 0 & 0 & 1 & 0 & 0 \\
					E3 & 0 & 1 & 0 & 0 & 1 & 0 & 0 & 1 & 0 & 0 & 0 & 1 & 0 & 0 & 1 \\
					E4 & 0 & 1 & 0 & 0 & 1 & 0 & 0 & 1 & 0 & 0 & 0 & 1 & 0 & 0 & 1 \\
					E5 & 0 & 1 & 0 & 0 & 1 & 0 & 0 & 1 & 0 & 0 & 0 & 1 & 0 & 0 & 1 \\
					\hline
				\end{tabular}}
				\caption{Adjacency matrix of the $\Phi_\omega$-maximizing graph}\label{tab2}
			\end{table}
			\begin{table}[ht]
				\centering
				\tiny
				\setlength{\tabcolsep}{2.5pt}
				\begin{tabular}{r|*{14}{c}}
					& A2 & A3 & B4 & B5 & C3 & C4 & D0 & D1 & D2 & D3 & E0 & E1 & E2 & E3 \\ \hline
					A2 & 0 & 0 & 0 & 0 & 0 & 0 & 1 & 0 & 0 & 0 & 1 & 0 & 0 & 0 \\
					A3 & 0 & 0 & 0 & 0 & 0 & 0 & 1 & 0 & 0 & 0 & 1 & 0 & 0 & 0 \\
					B4 & 1 & 0 & 1 & 0 & 0 & 0 & 0 & 0 & 0 & 0 & 0 & 0 & 0 & 0 \\
					B5 & 1 & 0 & 1 & 0 & 0 & 0 & 0 & 0 & 0 & 0 & 0 & 0 & 0 & 0 \\
					C3 & 0 & 0 & 0 & 0 & 1 & 0 & 0 & 0 & 1 & 0 & 0 & 0 & 1 & 0 \\
					C4 & 0 & 0 & 0 & 0 & 1 & 0 & 0 & 0 & 1 & 0 & 0 & 0 & 1 & 0 \\
					D0 & 0 & 0 & 0 & 0 & 0 & 0 & 0 & 0 & 0 & 1 & 0 & 0 & 0 & 1 \\
					D1 & 0 & 0 & 0 & 0 & 0 & 0 & 0 & 0 & 0 & 1 & 0 & 0 & 0 & 1 \\
					D2 & 0 & 0 & 0 & 0 & 0 & 0 & 0 & 0 & 0 & 1 & 0 & 0 & 0 & 1 \\
					D3 & 0 & 0 & 0 & 0 & 0 & 0 & 0 & 0 & 0 & 1 & 0 & 0 & 0 & 1 \\
					E0 & 0 & 1 & 0 & 1 & 0 & 1 & 0 & 1 & 0 & 0 & 0 & 1 & 0 & 0 \\
					E1 & 0 & 1 & 0 & 1 & 0 & 1 & 0 & 1 & 0 & 0 & 0 & 1 & 0 & 0 \\
					E2 & 0 & 1 & 0 & 1 & 0 & 1 & 0 & 1 & 0 & 0 & 0 & 1 & 0 & 0 \\
					E3 & 0 & 1 & 0 & 1 & 0 & 1 & 0 & 1 & 0 & 0 & 0 & 1 & 0 & 0 \\
					\hline
				\end{tabular}
				\caption{Adjacency matrix of the $\Phi_\omega$-minimizing graph}\label{tab3}
			\end{table}
Then $h_{top}(\overline{X}(\Phi_\omega))$	and $h_{top}(\underline{X}(\Phi_\omega))$ are the logarithm of the Perron-Frobenius eigenvalue of the adjacency matrices, namely $\log 3$.
It follows that the supremal H\"older exponent of inverse conjugacies $h^{-1}_{t\omega}$ are bounded from below by
\[\frac{h_{top}(\overline{X}(\Phi_\omega))}{h_{top}(X)}= \frac{h_{top}(\underline{X}(\Phi_\omega))}{h_{top}(X)}=\frac{\log 3}{\log\theta_0}\approx 0.64.\]

			\begin{figure}[h!]
%				\centering
				\includegraphics[width=0.7\linewidth]{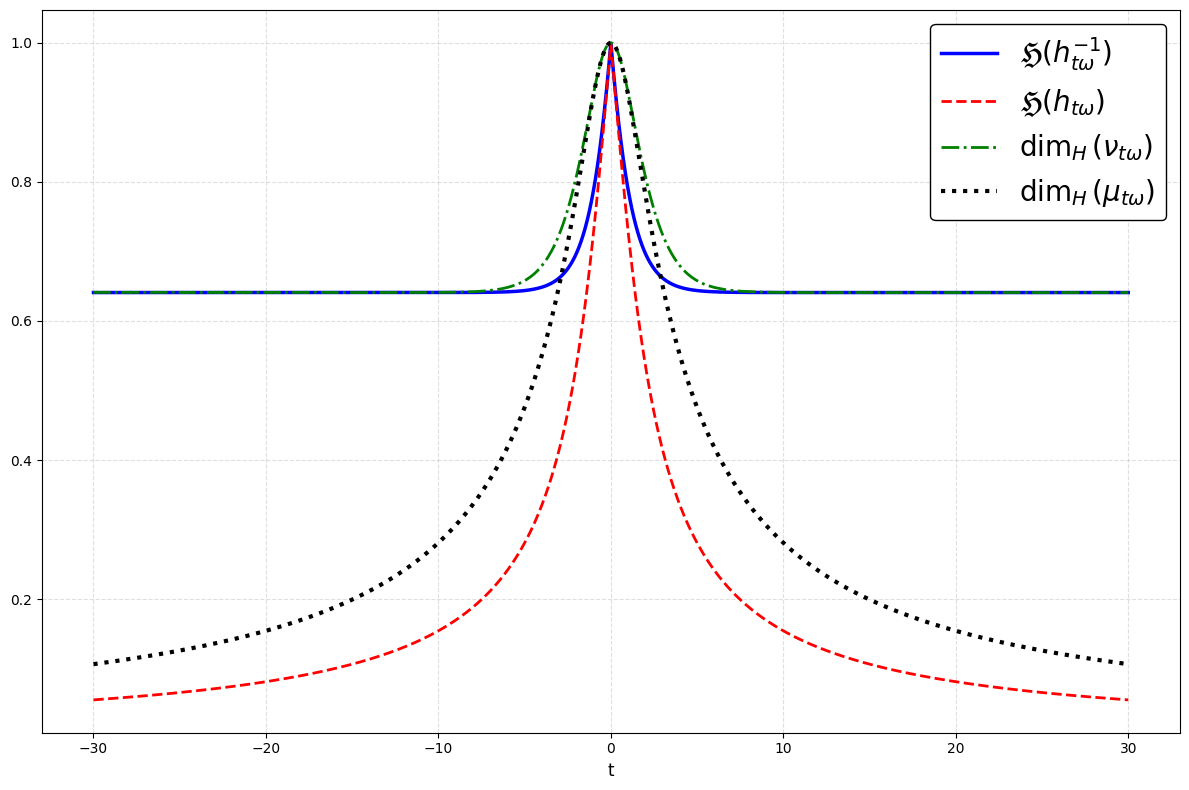}
				\caption{Plots showing the behavior of Hausdorff dimensions and H\"older exponents as the parameter $t$ varies}\label{fig:placeholder}
			\end{figure}
To conclude, in Figure~\ref{fig:placeholder} we present plots showing the dependence of the Hausdorff dimensions and the supremal H\"older exponents on the parameter $t$.
The reader may easily notice the symmetry of all the graphs, which is not accidental. By applying fairly elementary arguments, this symmetry follows from the symmetry of the permutation $\pi$ defining the IET $T$.		

\section*{Acknowledgements}
\noindent
The first author was partially supported by the Narodowe Centrum Nauki Grant 2025/57/B/ST1/00704.
The second author was partially supported by the Narodowe Centrum Nauki Grant   2023/50/O/ST1/00045.

		\end{document}